\documentclass[11p,leqno]{amsart}
\usepackage{amsmath}
\usepackage{amsfonts}
\usepackage{amssymb}
\usepackage{graphicx}
\usepackage{color}
\usepackage[shortalphabetic]{amsrefs}
\newtheorem{theorem}{Theorem}[section]
\newtheorem*{theorem*}{Theorem}

\newtheorem{corollary}[theorem]{Corollary}
\newtheorem{proposition}{Proposition}[section]

\theoremstyle{definition}

\theoremstyle{remark}
\newtheorem{remark}[theorem]{Remark}

\newcommand{\RR}{\mathbb{R}}
\newcommand{\E}{\mathrm{e}}

\numberwithin{equation}{section}

\begin{document}
\title[Modulus of continuity]{\bf Eigenvalue comparison on Bakry-Emery manifolds}

\author{Ben Andrews}
\address{Mathematical Sciences Institute, Australia National University; Mathematical Sciences Center, Tsinghua University; and Morningside Center for Mathematics, Chinese Academy of Sciences.}
\email{Ben.Andrews@anu.edu.au}
\thanks{Partially supported by Discovery Projects grant DP0985802 of the Australian Research Council.}
\author{Lei Ni}
\address{Department of Mathematics, University of California at San Diego, La Jolla, CA 92093, USA}
\email{lni@math.ucsd.edu}
\thanks{Partially supported by NSF grant DMS-1105549.}

\maketitle

\section{A lower bound for the first eigenvalue of the drift Laplacian}

Recall that  $(M, g, f)$, a triple consisting of a manifold $M$, a Riemannian metric $g$ and a smooth function $f$,   is called a  gradient Ricci soliton if the Ricci curvature and the Hessian of $f$ satisfy:
\begin{equation}\label{sgs}
{\rm Rc}_{ij}+f_{ij}=a g_{ij}.
\end{equation}
It is called shrinking, steady, or expanding soliton if $a>0$, $a=0$ or $a<0$ respectively.
In this paper we apply the modulus of continuity estimates developed in \cites{AC1,AC2, AC3} to give an eigenvalue estimate on gradient  solitons for the operator $\Delta_f\doteqdot \Delta -\langle \nabla(\cdot), \nabla f \rangle$ on strictly convex $\Omega\subset M$ with diameter $D$ and smooth boundary. In fact the result works for manifolds with lower bound on the  so-called Bakry-Emery Ricci tensor, namely ${\rm Rc}_{ij}+f_{ij}\ge a g_{ij}$ for some $a\in \mathbb{R}$. An earlier result of this kind was obtained in \cite{FS} for shrinking solitons.

In this  section, we extend a comparison theorem of \cite{AC3} on the modulus of continuity to manifolds with lower bound on the co-called Bakry-Emery Ricci tensor. The eigenvalue comparison result for manifolds with lower bound on the Bakry-Emery tensor generalizes the earlier lower estimates of Payne-Weinberger\cite{PW}, Li-Yau\cite{LY} and Zhong-Yang\cite{ZY}. A consequence of this is a lower diameter estimate for nontrivial gradient shrinking solitons, which improves \cite{FS} with a different approach and a rather short argument.  We remark here that the eigenvalue estimate we obtain is sharp for $(M,g,f)$ satisfying the Bakry-Emery-Ricci lower bound ${\rm Rc}_{ij}+f_{ij}\geq ag_{ij}$, but presumably is not so for Ricci solitons where the Bakry-Emery-Ricci tensor is constant, and so we expect that our diameter bound is also not sharp.  We discuss the sharpness of the eigenvalue inequality in section \ref{sec:sharp}.

Before we state the result, we define a corresponding $1$-dimensional eigenvalue problem. On $[-\frac{D}{2}, \frac{D}{2}]$ we consider the functionals
$$\mathcal{F}(\psi)=\int_{\frac D 2}^{\frac D 2} \E^{-\frac{a}{2} s^2} (\psi')^2\, ds, \quad \text{ and } \quad    \mathcal{R}(\psi)=\frac{\mathcal{F}(\psi)}{\int \E^{-\frac{a}{2}s^2}\psi^2ds}, $$  namely the Dirichlet energy with weight $\E^{-\frac{a}{2} s^2}$ and its Rayleigh quotient. The  associated elliptic operator is  $\mathcal{L}_{a}=\frac{d^2}{ds^2}-a\, s \frac{d}{ds}$. Let $\bar{\lambda}_{a,D}$ be the first non-zero Neumann eigenvalue of $\mathcal{L}_{a}$, which is the minimum of $\mathcal{R}$ among $W^{1, 2}$-functions with zero average.

\begin{theorem}\label{sgs-thm1} Let $\Omega$ be a compact manifold $M$,  or a bounded strictly convex domain inside a complete manifold $M$, satisfying that ${\rm Rc}_{ij}+f_{ij}\ge a\, g_{ij}$. Assume that $D$ is the diameter of $\Omega$. Then the first non-zero Neumann eigenvalue $\widetilde{\lambda}_1$ of the operator $\Delta_f$ is at least $\bar{\lambda}_{a,D}$.
\end{theorem}
\begin{proof} First we extend Theorem 2.1 of \cite{AC3} to this setting.  Recall that $\omega$ is a modulus of continuity for a function $f$ on $M$ if for all $x$ and $y$ in $M$, $|f(y)-f(x)|\leq 2\omega\left(\frac{d_M(x,y)}{2}\right)$.

\begin{proposition}\label{ex-AC-com}Let $v(x,t)$ be  a solution to
 \begin{equation}\label{heat-drift}
\frac{\partial v}{\partial t}=\Delta v-2\langle X, \nabla v\rangle
\end{equation}
  with $2X=\nabla f$. Assume also $v(x,t)$ satisfies the Neumann boundary condition. Suppose that $v(x, 0)$ has a modulus of continuity $\varphi_0(s): [0, \frac{D}{2}]\to \mathbb{R}$ with $\varphi_0(0)=0$ and $\varphi_0'>0$ on $[0, \frac{D}{2}]$. Assume further that there exists a
function $\varphi(s, t):[0, \frac{D}{2}]\times \mathbb{R}_+\to \mathbb{R}$ such that
 \begin{itemize}
\item[(i)] $\varphi(s,0) = \varphi_0(s)$ on $ [0,D/2]$;
\item[(ii)] $\frac{\partial \varphi}{\partial t}\ge \varphi'' -a\, s\, \varphi'$ on
$[0, D/2]\times \mathbb{R}_+$;
\item[(iii)] $\varphi'(s,t)>0$ on $[0, \frac{D}{2}]$;
\item
[(iv)] $\varphi(0, t) \ge 0$ for each $t\ge 0$.
\end{itemize}
Then $\varphi(s, t)$ is a modulus of the continuity of $v(x, t)$ for $t>0$.
 \end{proposition}

 The proof to the proposition is a modification of the argument to Theorem 2.1 of \cite{AC3}. Precisely, consider $$
 \mathcal{O}_\epsilon(x, y, t)\doteqdot v(y, t)-v(x, t)-2\varphi\left(\frac{r(x, y)}{2}, t\right)-\epsilon \E^t
  $$
  and it suffices to prove that the maximum of $\mathcal{O}_\epsilon(\cdot, \cdot, t)$ is non-increasing in $t$. The strictly convexity, the Neumann boundary condition satisfied by $v(x, t)$, and the positivity of $\varphi'$ rule out the possibility that the maximum can be attained at $(x_0, y_0)\in \partial (\Omega\times \Omega)$. For the interior pair $(x_0, y_0)$ where the maximum of  $\mathcal{O}_\epsilon(\cdot, \cdot, t)>0$ is attained, pick a frame $\{e_i\}$ as before at $x_0$ and parallel translate it along a minimizing geodesic $\gamma(s):[0, d]\to M$ joining $x_0$ with $y_0$. Still denote it by $\{e_i\}$. Let $\{E_i\}$ be the frame at $(x_0, y_0)$ (in $T_{(x_0, y_0)}\Omega\times \Omega$) as in Section 2. Direct calculations show that at $(x_0, y_0)$,
  \begin{eqnarray*}
  \left(\frac{\partial}{\partial t}-\sum_{j=1}^n \nabla^2_{E_j E_j} \right)\mathcal{O}_\epsilon (x, y, t)
  &=& -\left(\langle \nabla f(y), \gamma'\rangle-\langle \nabla f(x), \gamma'\rangle \right) \varphi' \\
  &\,&+\varphi'\sum_{i=1}^{n-1}\nabla^2_{E_i E_i} r(x, y) -2\varphi_t+2\varphi''-\epsilon \E^t.
  \end{eqnarray*}
Here we have used the first variation $\nabla \mathcal{O}(\cdot, \cdot, t)=0$ at $(x_0, y_0)$ which implies the identities
$$
(\nabla v)(y, t)=\varphi' \, \gamma'(d),\quad \quad  (\nabla v)(x, t)=\varphi'\,  \gamma'(0).
$$
Now choose the variational vector field $V_i(s)=e_i(s)$, the parallel transport of $e_i$ along $\gamma(s)$, along $\gamma(s)$, the second variation computation gives  that
$$
\sum_{i=1}^{n-1}\nabla^2_{E_i E_i} r(x, y) \le -\int_0^d \operatorname{Rc}(\gamma', \gamma')\, ds.
$$
Hence at $(x_0, y_0)$ we have that
\begin{eqnarray*}
 \left(\frac{\partial}{\partial t}-\sum_{j=1}^n \nabla^2_{E_j E_j} \right)\mathcal{O}(x, y, t)
&\le& -\varphi'\int_0^d (\nabla^2 f+\operatorname{Rc})(\gamma', \gamma')\, ds -2\varphi_t+2\varphi''\\
&\le& -\varphi' a r(x, y)-2\varphi_t+2\varphi''\\
&\le & 0.
\end{eqnarray*}
Here we have used $d=r(x, y)$ and $s=\frac{r(x, y)}{2}$.
This is enough to prove the proposition.

To prove the theorem, let $\bar{\omega}(s)$ be the first non-constant eigenfunction for $\mathcal{L}_{a}$, which can be chosen to be positive on $(0, \frac{D}{2})$. To apply the proposition we consider $\bar{\omega}^{D'}(s)$ be the eigenfunction on $[-\frac{D'}{2},\frac{D'}{2}]$ with the corresponding eigenvalue $\bar{\lambda}_a^{D'}$. Let $\varphi(x, t)=C \E^{-\bar{\lambda}^{D'}_a t} \bar{\omega}^{D'}(s)$. Let $w(x)$ be the first non-constant eigenfunction of $\Delta_f$ and let $v(x, t)=e^{-\widetilde{\lambda}_1 t}w(x)$. Since $\bar{\omega}^{D'}(s)$ is an odd function (by adding an eigenfunction $\psi(s)$ with $\psi(-s)$ one can always obtain one), we do have $\varphi(0, t)=0$. The possibility of choosing $(\bar{\omega}^{D'})'(s)> 0$ on $[0,\frac{D}{2}]$ can be proved as follows. By the uniqueness, we have that  $(\bar{\omega}^{D'})'(0)>0$. It suffices to show that  $(\bar{\omega}^{D'})'(s)>0$ for $s\in (0, \frac{D'}{2})$.
 Also observe that $\bar{\omega}^{D'}(\frac{D'}{2})>0$. By the ODE $\mathcal{L}_a\bar{\omega}^{D'}=-\bar{\lambda} \bar{\omega}^{D'}$ we can conclude that $(\bar{\omega}^{D'})'>0$ on $[\frac{D'}{2}-\epsilon, \frac{D'}{2})$. Let $y(s)=(\bar{\omega}^{D'})'(s)$.
The ODE also  forces $y>0$ for $s\in (0, \frac{D'}{2})$ since otherwise we assume $s_1<\frac{D'}{2}$ is the biggest zero. Note that $\bar{\omega}^{D'}(s_1)<0$ and it is strictly convex near $s_1$. Now clearly near $s_1$ one can raise the value of $\bar{\omega}^{D'}$ by replacing part of the graph with a line interval parallel to the $x$-axis, hence lower the energy $\mathcal{F}$. This contradicts  the fact that  $\bar{\lambda}$ is the minimum of the quotient $\mathcal{R}(\psi)$   among all nonzero $W^{1,2}(\E^{-\frac{a}{2}s^2}ds)$ function with zero average.

 Finally as before the proposition implies that for sufficient large $C$, $\varphi(s, t)$ is a modulus of continuity of $v(x, t)$. Hence $\widetilde{\lambda}_1\ge \bar{\lambda}_{a,D'}$. The claimed result follows by letting $D'\to D$.
 \end{proof}

\section{Sharpness of the lower bound}\label{sec:sharp}

In this section we show that (for $n\geq 3$ for any $a$ or for $n\geq 2$ for $a\leq 0$) the lower bound $\tilde\lambda_1\geq \bar\lambda_a$ given in Theorem \ref{sgs-thm1} is sharp:  Precisely, for each $\varepsilon>0$ we construct a Bakry-Emery manifold $(M,g,f)$ with diameter $D$ and $\tilde\lambda_1<\bar\lambda_{a,D}+\varepsilon$.

We will construct a smooth manifold $M$ which is approximately a thin cylinder with hemispherical caps at each end.  Let $\gamma$ be the curve in $\RR^2$ with curvature $k$ given as function of arc length as follows for suitably small positive $r$ and $\delta>0$ small compared to $r$:
\begin{equation}\label{eq:k}
k(s) = \begin{cases}
\frac{1}{r},&s\in\left[0,\frac{\pi r}{2}-\delta\right];\\
\varphi\left(\frac{s-\frac{\pi r}{2}}{\delta}\right)\frac1r,&s\in\left[\frac{\pi r}{2}-\delta, \frac{\pi r}{2}+\delta\right];\\
0,&s\in\left[\frac{\pi r}{2}+\delta,\frac{D}{2}\right],
\end{cases}
\end{equation}
extended to be even under reflection in both $s=0$ and $s=D/2$.  This corresponds to a pair of line segments parallel to the $x$ axis, capped by semicircles of radius $r$ and smoothed at the joins.  We write the corresponding embedding $(x(s),y(s))$.
Here $\varphi$ is a smooth nonincreasing function with $\varphi(s)=1$ for $s\leq -1$, $\varphi(s)=0$ for $s\geq 1$, and satisfying $\varphi(s)+\varphi(-s)=1$.
We choose the point corresponding to $s=0$ to have $y(0)=0$ and $y'(0)=1$.  The manifold $M$ will then be the hypersurface of rotation in $\RR^{n+1}$ given by
$\{(x(s),y(s)z):\ s\in\RR,\ z\in \mathbb{S}^{n-1}\}$.  On $M$ we choose the function $f$ to be a function of $s$ only, such that
\begin{equation}\label{eq:f}
f''=\begin{cases}
a\left(1-\frac{D}{\pi r}\right),& s\in\left[0,\frac{\pi r}{2}-\delta\right];\\
\varphi\left(\frac{s-\frac{\pi r}{2}}{\delta}\right)a\left(1-\frac{D}{\pi r}\right)
+\left(1-\varphi\left(\frac{s-\frac{\pi r}{2}}{\delta}\right)\right)a,&s\in\left[\frac{\pi r}{2}-\delta, \frac{\pi r}{2}+\delta\right];\\
a,&s\in\left[\frac{\pi r}{2}+\delta,\frac{D}{2}\right],
\end{cases}
\end{equation}
with $f'(0)=0$ (the value of $f(0)$ is immaterial).  Note that this choice gives $f'(D/2)=0$.  We extend $f$ to be even under reflection in $s=0$ and $s=D/2$.

With these choices we compute the Bakry-Emery-Ricci tensor and verify that the eigenvalues are no less than $a$ for suitable choice of $r$.  The eigenvalues of the second fundamental form are
$k(s)$ (in the $s$ direction) and $\frac{\sqrt{1-(y')^2}}{y}$ in the orthogonal directions.  Therefore the Ricci tensor has eigenvalues $(n-1)k\frac{\sqrt{1-(y')^2}}{y}$ in the $s$ direction, and $k\frac{\sqrt{1-(y')^2}}{y}+(n-2)\frac{1-(y')^2}{y^2}$ in the orthogonal directions.
We can also compute the eigenvalues of the Hessian of $f$:  The curves of fixed $z$ in $M$ are geodesics parametrized by $s$, the the Hessian in this direction is just $f''$ as given above.  Since $f$ depends only on $s$ we also have that $\nabla^{2}f(\partial_{s},e_{i})=0$ for $e_{i}$ tangent to $\mathbb{S}^{n-1}$, and $\nabla^{2}f(e_{i},e_{j})=\frac{y'}{y}f'\delta_{ij}$.

The identities $y(s)=\int_0^s\cos\left(\theta(\tau)\right)\,d\tau$ and $y'(s) = \cos\left(\theta(s)\right)$ where $\theta(s)=\int_0^sk(\tau)\,d\tau$ applied to \eqref{eq:k} imply that
$$
y=\begin{cases}r\sin(s/r),&s\in\left[0,\frac{\pi r}{2}-\delta\right];\\
r(1+o(\delta)),&s\in\left[\frac{\pi r}{2}-\delta, \frac{D}{2}\right],
\end{cases}\qquad
y'=\begin{cases}
\cos(s/r),&s\in\left[0,\frac{\pi r}{2}-\delta\right];\\
o(\delta),&s\in\left[\frac{\pi r}{2}-\delta, \frac{\pi r}{2}+\delta\right];\\
0,&s\in\left[\frac{\pi r}{2}+\delta,\frac{D}{2}\right],
\end{cases}
$$
as $\delta$ approaches zero.
This gives the following expressions for the Bakry-Emery Ricci tensor $\text{\rm Rc}_f=\text{\rm Rc}+\nabla^2f$:
$$
\text{\rm Rc}_f(\partial_s,\partial_s)
=\begin{cases}
a+\frac{n-1}{r^2}-\frac{aD}{\pi r},&s\in\left[0,\frac{\pi r}{2}-\delta\right];\\
a+\varphi\left(\frac{2-\frac{\pi r}{2}}{\delta}\right)\left(\frac{n-1}{r^2}(1+o(\delta))-\frac{aD}{\pi r}\right)&s\in\left[\frac{\pi r}{2}-\delta, \frac{\pi r}{2}+\delta\right];\\
a,&s\in\left[\frac{\pi r}{2}+\delta,\frac{D}{2}\right],
\end{cases}
$$
and
$$
\text{\rm Rc}_f(e,e)=
\begin{cases}
\frac{n-1}{r^2}+\frac{as\left(1-\frac{D}{\pi r}\right)}{r\tan\left(s/r\right)},&s\in\left[0,\frac{\pi r}{2}-\delta\right];\\
\frac{n-2}{r^2}+o(\delta),&s\in\left[\frac{\pi r}{2}-\delta, \frac{\pi r}{2}+\delta\right];\\
\frac{n-2}{r^2}(1+o(\delta))&s\in\left[\frac{\pi r}{2}+\delta,\frac{D}{2}\right],
\end{cases}
$$
while $\text{\rm Rc}_f(\partial_s,e)=0$, for any unit vector $e$ tangent to $\mathbb{S}^{n-1}$.  In particular we have $\text{\rm Rc}_f\geq ag$ for sufficiently small $r$ and $\delta$ for any $a\in\RR$ if $n\geq 3$, and for $a<0$ if $n=2$.  Note also that the diameter of the manifold $M$ is $D(1+o(\delta))$.

Having constructed the manifold $M$, we now prove that for this example the first non-trivial eigenvalue $\tilde\lambda_1$ of $\Delta_f$ can be made as close as desired to $\bar\lambda_{a,D}$ by choosing $r$ and $\delta$ small.   Theorem \ref{sgs-thm1} gives the upper bound $\tilde\lambda_1\geq \bar\lambda_{a,D(1+o(\delta))}=\bar\lambda_{a,D}+o(\delta)$.  To prove an upper bound we can simply find a suitable test function to substitute into the Rayleigh quotient which defines $\tilde\lambda_1$:  We set
$$
\psi(s,z) = \begin{cases}w(s-D/2),& \frac{\pi r}{2}+\delta\leq s\leq D-(\frac{\pi r}{2}+\delta);\\
w\left(\frac{D}{2}-\frac{\pi r}{2}-\delta\right),& 0\leq s\leq \frac{\pi r}{2}+\delta, \mbox{ and } D-\frac{\pi r}{2}-\delta \le s\le D.
\end{cases}
$$
where $w$ is the solution of $w''-asw'+\bar\lambda_{a,D-\pi r-2\delta}w=0$ with $w(0)=0$ and $w'\left(\frac{D}{2}-\frac{\pi r}{2}-\delta\right)=0$ and $w'(0)=1$.  This choice gives
$$
{\mathcal R}(\psi) = \frac{
\bar\lambda_{a,D-\pi r-2\delta}\int_{\{|s-D/2|\leq \frac{D}{2}-\frac{\pi}{2r}-\frac{\delta}{2}\}}\psi^2\E^{-f}dVol(g)
}{
\int_{\{|s-D/2|\leq D/2\}}\psi^2\E^{-f}dVol(g)
}\leq \bar\lambda_{a,D-\pi r-2\delta}.
$$
It follows that $\tilde\lambda_1\to\bar\lambda_{a,D}$ as $r$ and $\delta$ approach zero, proving the sharpness of the lower bound in Theorem \ref{sgs-thm1}.

\begin{remark}
If we allow manifolds with boundary the construction is rather simpler:  Simply take a cylinder $r\mathbb{S}^{n-2}\times[-D/2,D/2]$ for small $r$, with quadratic potential $f=\frac{a}{2}s^2$, and substitute the test function $\psi(z,s)=w(s)$ defined above.
\end{remark}

\section{A linear lower bound}\label{sec:linear}

Concerning the lower estimate of $\bar{\lambda}_a$, at least for $a\ge 0$, note that  $y=(\bar{\omega}^{D'})'$ satisfies that
$$
\left(\E^{-\frac{a}{2} s^2}y'\right)'=(a-\bar{\lambda}_a)\E^{-\frac{a}{2} s^2}y.
$$
This together with the maximum principle applying to $y^2$ implies that $\bar{\lambda}_a\ge a$. Applying Proposition \ref{ex-AC-com} to the trivial case $M=\mathbb{R}$ with $\varphi(x, t)=\E^{-(\frac{\pi}{D'})^2t} \sin (\frac{\pi}{D'} s)$, and letting $D'\to D$ we also get $\bar{\lambda}_a\ge \frac{\pi^2}{D^2}$.

On the other hand, a normalization procedure reduces the problem of finding/estimation of the first nontrivial Neumann eigenvalue for the involved ODE to finding the first nontrivial Neumann eigenvalue $\bar{\lambda}_{2, \sqrt{\frac{a}{2}} D}$ for the Hermite equation: $\frac{d^2}{ds^2}-2s\frac{d}{ds}$ on the interval $[-\sqrt{\frac{a}{2}}\frac{D}{2}, \sqrt{\frac{a}{2}}\frac{D}{2}]$ since $\bar{\lambda}_a=\frac{a}{2}\bar{\lambda}_{2, \sqrt{\frac{a}{2}} D}$. The Neumann eigenvalue for the Hermite equation is then related to the  eigenvalue of the harmonic oscillator: $\frac{d^2}{ds^2}-s^2 $ with a certain  Robin boundary condition. The following result and its consequence improve the main results of \cite{FS}.

\begin{proposition}\label{ODE1} When $a>0$, the first nonzero Neumann eigenvalue $\bar{\lambda}_a$ is bounded from below by $\frac{a}{2}+\frac{\pi^2}{D^2}$. In particular $\widetilde{\lambda}_1(\Omega)$, with $\Omega$ being a convex domain in any Riemannian manifold with ${\rm Rc}_{ij}+f_{ij}\ge a g_{ij}$, is bounded from below by $\frac{a}{2}+\frac{\pi^2}{D^2}$.
\end{proposition}
\begin{proof} By the above renormalization procedure, it is enough to prove the result for the operator
$\frac{d^2}{ds^2}-2s\frac{d}{ds}$ on interval $[-\frac{D}{2}, \frac{D}{2}]$. Let $\phi$ be the first eigenfunction which is  odd. Let $y=\phi'$ and denote $\lambda$ the first (nonzero) Neumann value. Then direct calculation shows that
$$
\left(\E^{-s^2}y'\right)'=-(\lambda-2)\E^{-s^2 }y.
$$
Multiply $y$ on both sides  of the above equation and integrate the resulting equation on $[-\frac{D}{2}, \frac{D}{2}]$.
The fact $y=0$ on the boundary implies that
$$
\int_{-D/2}^{D/2} \E^{-s^2}(y')^2=(\lambda-2)\int_{-D/2}^{D/2}  \E^{-s^2}y^2.
$$
In the view that $y$ vanishes on the boundary, it implies that
$\lambda-2\ge \lambda_0$, the first Dirichlet eigenvalue of the operator $\frac{d^2}{ds^2}-2s\frac{d}{ds}$. Now we may introduce the the tranformation $w=\E^{-\frac{s^2}{2}}\varphi$. Direct calculation shows that  $\varphi$ is the first Dirichlet eigenfunction if and only if
$$
\frac{d^2}{ds^2} w- s^2 w=-(\lambda_0+1)w
$$
with $w$ vanishes on the boundary. By  Corollary 6.4 of \cite{N-expandMo} we have that
$$
\lambda_0+1\ge \frac{\pi^2}{D^2}.
$$
Combining them together we have that $\lambda\ge 1+\frac{\pi^2}{D^2}$. Scaling will give the claimed result.
\end{proof}

\begin{corollary}\label{con-FS-improve} If $(M, g, f)$ is a nontrivial gradient shrinking soliton satisfying (\ref{sgs}) with $a>0$. Then
$$
\operatorname{Diameter}(M, g)\ge \sqrt{\frac{2}{3a}}\pi.
$$
\end{corollary}
\begin{proof}
The result follows from the above lower estimate on the first Neumann eigenvalue, applying to the case that $\Omega=M$,  and the observation, Lemma 2.1 of \cite{FS}, that $2a$ is an eigenvalue of the operator $\Delta -\langle \nabla f, \nabla (\cdot)\rangle$.
\end{proof}

This result clearly is not sharp.   A better eigenvalue lower bound (and hence a better diameter lower bound) will follow from a better understanding of the first Dirichlet eigenvalue of the harmonic oscillator.  We investigate this in the next section.

We  remark that Proposition \ref{ODE1} also implies that for a compact Riemannian manifold $(M, g)$ satisfying ${\rm Rc}\ge (n-1)K$ for some $K>0$, the estimate $\lambda_1(M)\ge \frac{n-1}{2}K+\frac{\pi^2}{D^2}$ holds with $D$ being its diameter. This improves the earlier corresponding works in \cite{Ln}, \cite{Ya}, etc.

\section{The harmonic oscillator on bounded symmetric intervals}
\label{sec:EHO}
In this section we will investigate the sharp lower bound given by the eigenvalue of the one-dimensional harmonic oscillator on a bounded symmetric interval:   Recall from section \ref{sec:linear} that the first Neumann eigenvalue $\bar\lambda_{2,D}$ is equal to $\hat\lambda_{1,D}+1$, where $\hat\lambda_{b,D}$ is defined by the existence of a solution of the eigenvalue problem
\begin{align*}
w''+\left(\hat\lambda_{b,D}-bs^2\right)w&=0,\ s\in[-D/2,D/2];\\
w(D/2)=w(-D/2) &=0;\\
w(x)&>0,\ s\in(-D/2,D/2).
\end{align*}
The solution of the ordinary differential equation $w''-s^2w+\lambda w=0$ (with $w'(0)=0$), which is also called Weber's equation,  can be written in terms of confluent hypergeometric functions:  We have
$$
w(s) = \E^{-s^2}U\left(\frac14-\frac{\lambda}{8},\frac12,2s^2\right)
$$
where $U$ is the confluent hypergeometric function of the first kind.  Thus $\hat\lambda_{1,D}$ is the first root of the equation $U\left(\frac14-\frac{\lambda}{8},\frac12,\frac{D^2}{2}\right)=0$.  Since $U$ is strictly monotone in the first argument, the solution is an analytic function of $D$.

Noting that $\hat\lambda_{1,D}=\frac{\pi^2}{D^2}\hat\lambda_{\frac{D^4}{\pi^4},\pi}$,
we use a perturbation argument to compute the Taylor expansion for $\hat\lambda_{b,\pi}$ as a function of $b=\frac{a^2}{4}$ about $b=0$ (this provides an expansion for $\hat\lambda_{1,D}$ about $D=0$).
That is, we consider the solution $u$ of the eigenvalue problem
\begin{align*}
u''+\left(\lambda-bs^2\right)u&=0,\ s\in[-\pi/2,\pi/2];\\
u(\pi/2)=u(-\pi/2)&=0;\\
u(s)&>0,\ s\in(-\pi/2,\pi/2).
\end{align*}
The solution for $b=0$ is of course given by $u(s)=\cos(s)$.  The perturbation expansion produces a solution of the form
$$
u(s,b) = \sum_{k=0}^\infty b^k\sum_{j=0}^{3k}\left(\alpha_{k,j}s^j\cos s+\beta_{k,j}s^k\sin s\right),
$$
with $\lambda=\sum_{k=0}^\infty\lambda_kb^k$.  This expansion is unique provided we specify that $u$ is even, $\alpha_{0,0}=1$, $\beta_{0,1}=0$, and $\alpha_{k,0}=\beta_{k,0}=0$ for $k>0$.  The first few terms in the expansion for $\lambda$ are given by
\begin{align*}
\hat\lambda_{b,\pi} &= 1+\left(\frac{\pi^2}{12}-\frac12\right)b
+\left(\frac{\pi^4}{720}-\frac{5\pi^2}{48}+\frac78\right)b^2
+\left(\frac{\pi^6}{30240}-\frac{\pi^4}{48}+\frac{31\pi^2}{32}-\frac{121}{16}
\right)b^3\\
&\quad\null+\left(\frac{\pi^8}{362880}-\frac{\pi^6}{270}+\frac{683 \pi^4}{1280}-\frac{14573 \pi^2}{768}+\frac{17771}{128}\right)b^4+O(b^5).
\end{align*}

\begin{figure}\label{fig:Weber}
\includegraphics[scale=0.5]{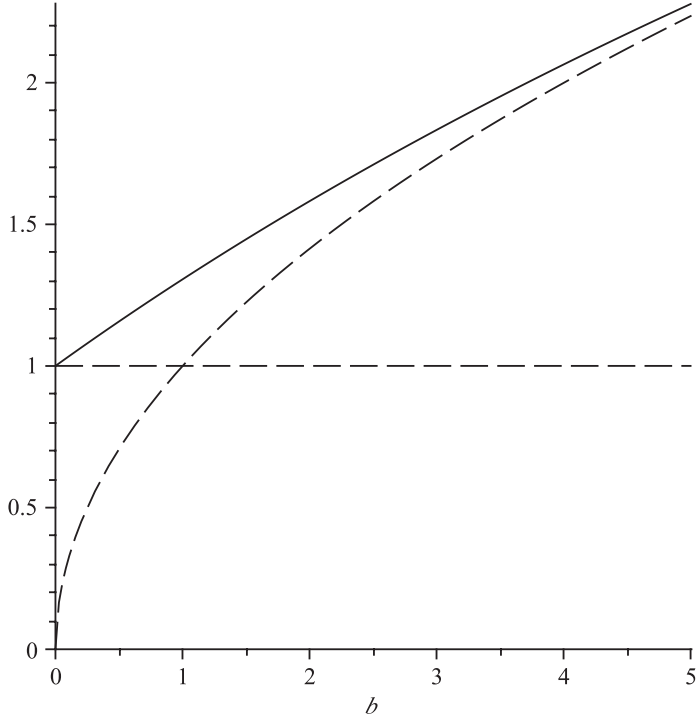}
\caption{The eigenvalue $\hat\lambda_{b,\pi}$ for Weber's equation $y''+(\lambda-bs^2)y=0$, $y(\pm\pi/2)=0$ [solid curve]; shown also are the lower bounds $\hat\lambda\geq 1$ and $\hat\lambda\geq\sqrt{b}$ [dashed curves].}
\end{figure}

We note that there is also a useful lower bound for $\hat\lambda_{b,\pi}$, which we can arrive at as follows:  The inclusion of $[-D/2,D/2]$ in $\RR$ implies $\hat\lambda_{1,D}\geq \lim_{d\to\infty}\hat\lambda_{1,d}=1$, with eigenfunction $u(s) = \E^{-s^2/2}$.  Therefore we also have
$$
\hat\lambda_{b,\pi}=\sqrt{b}\hat\lambda_{1,\pi b^{1/4}}\geq\sqrt{b}.
$$

This translates to an estimate for the drift eigenvalue $\bar\lambda_{a,\pi}$ appearing in Theorem \ref{sgs-thm1}:  We have $\bar\lambda_{a,\pi} = \frac{a}{2}+\hat\lambda_{a^2/4,\pi}$, giving the following Taylor expansion:
\begin{align*}
\bar\lambda_{a,\pi} &= 1+\frac{a}{2}+\left(\frac{\pi^2}{12}-\frac12\right)\frac{a^2}{4}
+\left(\frac{\pi^4}{720}-\frac{5\pi^2}{48}+\frac78\right)\frac{a^4}{16}
+\left(\frac{\pi^6}{30240}-\frac{\pi^4}{48}+\frac{31\pi^2}{32}-\frac{121}{16}
\right)\frac{a^6}{64}\\
&\quad\null+\left(\frac{\pi^8}{362880}-\frac{\pi^6}{270}+\frac{683 \pi^4}{1280}-\frac{14573 \pi^2}{768}+\frac{17771}{128}\right)\frac{a^8}{256}+O(a^{10}).
\end{align*}
In particular the lower bound $\hat\lambda\geq 1$ translates to $\bar\lambda\geq 1+a/2$, and the lower bound $\hat\lambda\geq\sqrt{b}$ translates to $\bar\lambda\geq a/2+\sqrt{a^2/4}=a$.  Finally, by scaling we obtain the following:
\begin{align*}
\bar\lambda_{a,D} &= \frac{\pi^2}{D^2}+\frac{a}{2}+\left(\frac{\pi^2}{12}-\frac12\right)\frac{D^2a^2}{4\pi^2}
+\left(\frac{\pi^4}{720}-\frac{5\pi^2}{48}+\frac78\right)\frac{D^6a^4}{16\pi^6}\\
&\quad\null
+\left(\frac{\pi^6}{30240}-\frac{\pi^4}{48}+\frac{31\pi^2}{32}-\frac{121}{16}
\right)\frac{D^{10}a^6}{64\pi^10}\\
&\quad\null+\left(\frac{\pi^8}{362880}-\frac{\pi^6}{270}+\frac{683 \pi^4}{1280}-\frac{14573 \pi^2}{768}+\frac{17771}{128}\right)\frac{D^{14}a^8}{256\pi^{14}}\\
&\quad\null+O(D^{18}a^{10})\quad\text{\rm as}\ D^2a\to 0.
\end{align*}

An interesting consequence of the Taylor  expansion (combined with the fact that the estimate $\tilde\lambda_1\geq \bar\lambda_{a,D}$ is sharp as proved in section \ref{sec:sharp}) is the following:

\begin{proposition}
The constant $a/2$ in the lower bound $\tilde\lambda_1\geq \frac{\pi^2}{D^2}+\frac{a}{2}$ is the largest possible.
\end{proposition}

This follows from the Taylor expansion for small values of $aD^2$.

We note that the sharp diameter bound (given by the value of $a$ where the dotted line $\lambda=2a$ intersects with the solid curve in Figure \ref{fig:drift}) is not dramatically different
from the one given in Corollary \ref{con-FS-improve} (where the dotted line intersects the dashed line $\lambda=1+a/2$).  Since the eigenvalue estimate $\tilde\lambda_1\geq \bar\lambda_{a,D}$ appears from the examples in section \ref{sec:sharp} to be sharp only in situations which are far from gradient solitons, we expect that neither of these diameter bounds is
close to the sharp lower diameter bound for a nontrivial gradient Ricci soliton.

\begin{figure}
\includegraphics[scale=0.5]{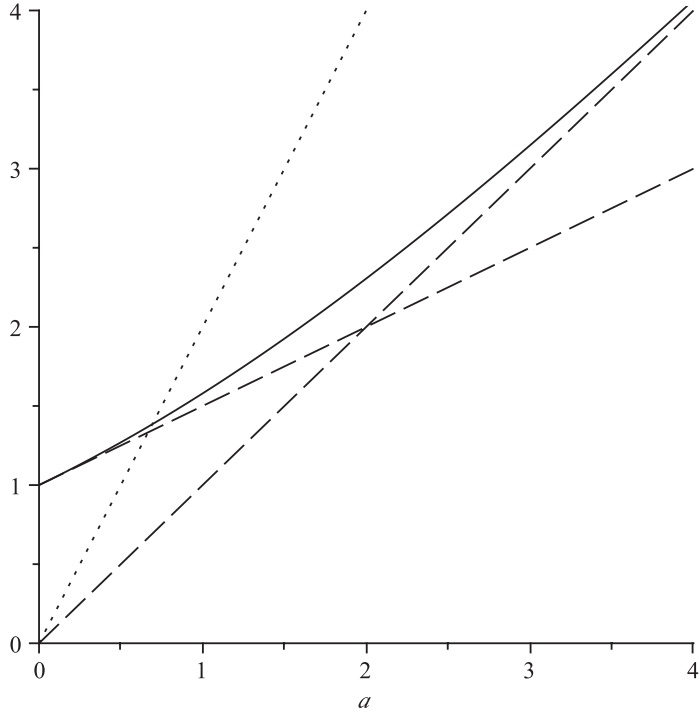}
\caption{The eigenvalue $\bar\lambda_{a,\pi}$ for the drift Laplacian equation $y''-asy'+\lambda y=0$, $y'(\pm\pi/2)=0$ [solid curve]; shown also are the lower bounds $\bar\lambda\geq 1+\frac{a}{2}$ and $\bar\lambda\geq a$ [dashed lines], and the line $\bar\lambda=2a$ corresponding to non-Einstein gradient Ricci solitons [dotted line].}\label{fig:drift}
\end{figure}

\end{document}